\documentclass[12pt]{amsart}

\usepackage{latexsym,amsfonts,amssymb,amsthm,amsmath,amscd,mathrsfs,color}

\newtheorem{theorem}{Theorem}[section]

\newtheorem{lemma}[theorem]{Lemma}
\newtheorem{proposition}[theorem]{Proposition}

\theoremstyle{definition}
\newtheorem{definition}[theorem]{Definition}
\newtheorem{remark}[theorem]{Remark}
\newtheorem{example}[theorem]{Example}

\numberwithin{equation}{section}

\newcommand{\diam}{\mathrm{diam}}

\allowdisplaybreaks

\begin{document}

%%%%% To ease editing, for IMPAN journals add:

\baselineskip=17pt

%%%%%%%%%%%

%% In the running head, replace first names by initials
%% and give an abbreviation of the title.

\title[On $s$-sets in spaces of homogeneous type]{On $s$-sets in spaces of homogeneous type}

% Information for first author
\author[M. Carena]{Marilina Carena}
\address{Instituto de Matem\'atica Aplicada del Litoral
(CONICET-UNL), Departamento de Matem\'{a}tica (FHUC-UNL),  Santa Fe,
Argentina.} \email{mcarena@santafe-conicet.gov.ar}
%Information for second author
\author[M. Toschi]{Marisa Toschi}
\address{Instituto de Matem\'atica Aplicada del Litoral
(CONICET-UNL), Departamento de Matem\'atica (FIQ-UNL), Santa Fe,
Argentina.} \email{mtoschi@santafe-conicet.gov.ar}

\date{}

%\thanks{The authors were supported by CONICET, CAI+D
%(UNL) and ANPCyT}

\subjclass[2010]{Primary  28A05; Secondary 28A78}

\keywords{space of homogeneous type, $s$-sets, Hausdorff measure}

%%% ----------------------------------------------------------------------
\maketitle
%%% ----------------------------------------------------------------------

\begin{abstract}
Let $(X,d,\mu)$ be a space of homogeneous type. In this note we study the relationship between
two types of $s$-sets:  relative to a distance  and relative to a measure. We find  a condition on a closed subset $F$
of $X$ under which we have that $F$ is $s$-set relative to the measure $\mu$ if and only if $F$ is $s$-set relative to $\delta$.
Here $\delta$ denotes the quasi-distance  defined by Mac\'ias and Segovia such that $(X,\delta,\mu)$ is a normal space.
In order to prove  this result,
 we show a covering type lemma and
 a type of Hausdorff measure based criteria for the $s$-set condition relative to $\mu$  of a given set.
\end{abstract}

\section{Introduction, notation and definitions}

A \emph{\textbf{quasi-metric}} on a set  $X$ is  a non-negative
 function  $d$ defined on $X\times X$ satisfying the following properties:
 \begin{enumerate}
\item $d(x,y)=0$ if and only if $x=y$;
\item $d(x,y)=d(y,x)$ for every $x,y\in X$;
\item there exists a constant
$K\geq 1$ such that  $d(x,y)\leq K(d(x,z)+d(z,y))$, for every $x,y,z\in X$.
\end{enumerate}
We will refer to $K$ as the triangle constant for $d$.
A quasi-distance $d$ on $X$ induces a topology  through the
neighborhood system given by the family of all subsets of $X$
containing a $d$-ball $B(x,r)=\{y\in X:d(x,y)<r\}$,
$r>0$ (see \cite{CoifWeiss}). In a quasi-metric space $(X,d)$ the \emph{\textbf{diameter}} of
a subset $E$ is defined as
\[\diam(E)=\sup\{d(x,y):x,y\in E\}.\]
%and the distance between a point $x\in X$ and a set $E$ is defined by $d(x,E)=\inf\{d(x,y): y\in E\}$.

Throughout this paper $(X,d)$ shall be a  quasi-metric
space such that the $d$-balls are open sets. Also we shall assume that $(X,d)$ has \emph{\textbf{finite metric dimension}}.
This means that there exists a constant $N\in\mathbb{N}$ such that
 any $d$-ball $B(x,2r)$ contains at most $N$ points  of any $r$-disperse
subset of $X$. A set $U$ is said to be  \emph{$\boldsymbol{r}$\textbf{-disperse}}
if $d(x,y)\geq r$ for every $x,y \in U$, $x\neq y$. If a quasi-metric space $(X,d)$ has finite metric dimension,
  every $r$-disperse  subset of $X$ has at most $N^m$ points in each
$d$-ball of radius $2^m r$, for all $m\in\mathbb N$ and every $r>0$ (see
\cite{CoifWeiss} and \cite{Asso}). Also it
 is well known that   every bounded subset $F$ of $X$ is totally bounded, so that for every  $r>0$ there exists a finite maximal $r$-disperse on $F$, whose cardinal depends on $\diam(F)$ and on~$r$. % An \emph{$\boldsymbol{r}$\textbf{-net}} is  a maximal $r$-disperse set. It is easy to check that
% $U$ is an $r$-net in $X$ if and only if $U$ is an $r$-disperse and  $r$-dense set in $X$, where \emph{$\boldsymbol{r}$\textbf{-dense}} means that for every $x\in X$ there exists $u\in U$ with $d(x,u)<r$.

We shall say that a closed subset $F$ of $X$ is \emph{\textbf{$\boldsymbol s$-set in $\boldsymbol{(X,d)}$}} with associated measure $\nu$,
 if $\nu$ is a  Borel measure  supported on $F$ such that
  \begin{equation}\label{eq: normal}
c^{-1}r^s\leq \nu(B(x,r))\leq cr^s,
\end{equation}
for every $x\in F$ and every $0<r<\diam(F)$, for some constant $c\geq 1$. When  the above conditions hold for every $0<r<r_0$, where $r_0$ is a positive number less than $\diam(F)$, we say that $F$  is \emph{\textbf{locally $\boldsymbol s$-set in $\boldsymbol{(X,d)}$}}.
In some references related to problems of harmonic analysis and partial differential equations, see for example \cite{ACDT}, this sets are called
\emph{\textbf{(locally) $\boldsymbol s$-Ahlfors}}.
In the bibliography belonging to geometric measure theory, such as \cite{Falconer}, an $s$-set $F$ is one for which $0< \mathscr H^s(F)<\infty$ where $\mathscr H^s$ is the Hausdorff measure of dimension $s$. Nevertheless, following \cite{Sjoding} we shall adopt the expression $s$-set  to name a set that supports a measure $\nu$ for which $\nu(B(x,r))$ behaves as $r^s$ for $r$ small.

In \cite{ACDT} is proved  that the concepts of $s$-set and locally $s$-set  coincide when the set $F$ is bounded and  $(X,d)$ has finite metric dimension.

We shall  now  recall the definitions of Hausdorff measure and Hausdorff dimension of a set in a quasi-metric space $(X,d)$.
 The basic aspects related to this concepts can be found in \cite{Falconer}.
For $\rho>0$, we say that a sequence $\{B_i=B(x_i,r_i)\}$ of subsets of $X$ is a \emph{$\boldsymbol\rho$\textbf{-cover}} by $d$-balls of a set $F$ if
$F\subseteq \bigcup B_i$ and $r_i\leq
\rho$ for every $i$. Let $F\subseteq X$ and $s\geq 0$ fixed.
We define
\[\mathscr{H}^s_{\rho}(F)=\inf\left\{\sum_{i=1}^{\infty} r_i^s: \{B_i\}
\textrm{ is a $\rho$-cover by $d$-balls of } F\right\}.\]
Clearly  $\mathscr{H}^s_{\rho}(F)$ increases when  $\rho$ decreces, so that the limit  when  $\rho$ tends to $0$ exists
 (although it may be infinite). Then we define
\[\mathscr{H}^s(F)=\lim_{\rho \to 0}\mathscr{H}^s_{\rho}(F)=\sup_{\rho>0}
\mathscr{H}^s_{\rho}(F).\]
We shall refer to  $\mathscr{H}^s(F)$ as the
 \emph{\textbf{Hausdorff measure}} of~$F$. The corresponding \emph{\textbf{Hausdorff dimension}} of $F$ is defined as
 $\dim_{\mathscr{H}}(F)=\inf\{s>0: \mathscr{H}^s(F)=0\}$. It is easy to see that any $s$-set  $F$ in $(X,d)$ satisfies that
 $\dim_{\mathscr{H}}(F)=s$ (see \cite{Sjoding}).

We shall point out that, if $(F,d)$ is  (locally) $s$-set, then there exists essentially only one Borel  measure $\nu$ satisfying the condition required in the definition. This fact is known in the Euclidean setting (see for instance \cite{Triebel}), and was proved for general quasi-metric spaces in \cite{ACDT}. More precisely, is proved that if $(X,d)$ has finite metric dimension and $F$ is (locally) $s$-set in $(X,d)$ with measure $\nu$, then $F$ is (locally) $s$-set en $(X,d)$ with the restriction of  $\mathscr H^s$ to~$F$.

A sufficient condition under which a quasi-metric space $(X,d)$ has finite metric dimension is when $X$ supports a doubling measure (see \cite{CoifWeiss}).
 A Borel measure $\mu$ defined  on the $d$-balls is said to be \emph{\textbf{doubling}} if for some constant $A\geq 1$ we have the inequality
\[0<\mu(B(x,2 r))\leq A\mu(B(x,r))<\infty,\]
for every $x\in X$ and every $r>0$. When $\mu$ is a doubling measure, we say that a point $x$ in $(X,d,\mu)$ is an \emph{\textbf{atom}} if $\mu(\{x\})>0$. When $\mu(\{x\})=0$ for every $x\in X$ we say that $\mu$ is a \emph{\textbf{non-atomic}} doubling measure. Mac\'ias and Segovia proved in \cite{M-S} that a point is an atom if and only if it is topologically isolated,
 and that the set of such points is at most countable. Throughout this paper we shall say that
 $(X,d,\mu)$ is a \emph{\textbf{space of homogeneous type}} if $\mu$ is a non-atomic doubling measure on the quasi-metric space $(X,d)$.

Given a  space of homogeneous type $(X,d,\mu)$, the Hausdorff measure and the Hausdorff dimension \emph{ relative to}  $\mu$ is consider in \cite{Sjoding}. Precisely, the \textbf{\emph{Hausdorff measure relative to $\boldsymbol \mu$}} is defined as ${H}^s(F):=\lim_{\rho \to 0}{H}^s_{\rho}(F)$, where
\[
{H}^s_{\rho}(F)=\inf\left\{\sum_{i=1}^{\infty}\mu^s(B_i): F\subseteq \bigcup_i B_i \textrm{ and } \mu(B_i)\leq\rho\right\},\]
where $B_i$ are $d$-balls on $X$. Then the \textbf{\emph{Hausdorff dimension relative to $\boldsymbol \mu$}} is defined by
\[\dim_H(F)=\inf\{s>0: H^s(F)=0\}.\]

These concepts conduce to give a definition of $s$-set relative to the measure $\mu$, compatible with $H^s$.
Given a space of homogeneous type  $(X,d,\mu)$, we shall say that a closed subset $F$ of $X$ is \emph{\textbf{$\boldsymbol s$-set in $\boldsymbol{(X,d,\mu)}$}} if there exist
a constant $c\geq1 $ and
 a Borel measure $m$ supported on $F$ such that
  \begin{equation}\label{eq: normal mu}
c^{-1}\mu(B(x,r))^s\leq m(B(x,r))\leq c\mu(B(x,r))^s,
\end{equation}
for every $x\in F$ and every $0<r<\diam(F)$.
As before,  if (\ref{eq: normal mu}) holds for every $0<r<r_0$, where $r_0$ is a positive number less than $\diam(F)$, we say that $F$  is \emph{\textbf{locally $\boldsymbol s$-set in $\boldsymbol{(X,d,\mu)}$}}.

It is now easy to see that each $s$-set $F$ in $(X,d,\mu)$ satisfies $\dim_H(F)=s$.\\

Given a space of homogeneous type $(X,d,\mu)$, in \cite{Sjoding} are also considered the concepts of $s$-sets, Hausdorff measure and Hausdorff dimension relative to a particular quasi-metric $\delta$ related to $(X,d,\mu)$. This quasi-metric was constructed by Mac\'ias and Segovia in \cite{M-S},
in such a way that the new structure $(X,\delta,\mu)$
becomes a normal space (in the sense that every $\delta$-ball in $X$ has $\mu$-measure equivalent to its ratio), and the topologies induced on $X$ by $d$ and $\delta$ coincide.
This quasi-metric is defined by
\[\delta(x,y)=\inf\{\mu(B): B \textrm{ is a $d$-ball with  $x,y\in B$}\}\]
if $x\neq y$, and $\delta(x,y)=0$ if $x=y$. %By the definition of $\delta$
%we have that for every  $x\in X$ and every $r>0$, if $\mu(\{x\})\geq r$ then
%$\bdelta(x,r)=\{x\}$,  where $B_\delta(x,r):=\{y\in X:\delta(x,y)<r\}$.
%
It will be also useful  to notice that in the proof of the above mentioned result of Mac\'ias
and Segovia it is proved that
\[B_\delta(x,r)=\bigcup\{B: B\textrm{ is a $d$-ball with  $x\in B$  and  $\mu(B)<r$}\},\]
for every $x\in X$ and every $r>0$, where $B_\delta(x,r):=\{y\in X:\delta(x,y)<r\}$ denotes
the ball in $X$ relative to $\delta$. Throughout  this paper  $\delta$ shall denote this quasi-metric.

Then, we can consider the concept of $s$-set in $(X,\delta)$, the Hausdorff measure
relative to $\delta$, and the corresponding Hausdorff dimension. More precisely, we shall denote
${G}^s(F):=\lim_{\rho \to 0}{G}^s_{\rho}(F)$, where
\[
{G}^s_{\rho}(F)=\inf\left\{\sum_{i=1}^{\infty}r_i^s: F\subseteq \bigcup_i B_\delta(x_i,r_i) \textrm{ and } r_i\leq\rho\right\},\]
and
\[\dim_G(F)=\inf\{s>0: G^s(F)=0\}.\]
In \cite[Propo. 1.5]{Sjoding} is proved that $H^s(F)$ and $G^s(F)$ are equivalent, and then $\dim_H(F)=\dim_G(F)$ for any
subset $F$ of $X$. %On the other hand, as we already mentioned, any $s$-set $F$ in $(X,d,\mu)$ satisfies $\dim_H(F)=s$,
% and any $s$-set $F$ in $(X,\delta)$ satisfies $\dim_G(F)=s$.
In this note we explore
 the relationship between the concepts of $s$-set in $(X,d,\mu)$  and  $s$-set
in $(X,\delta)$. \\%We prove that if $F$ in a (locally)  $s$-set
%in $(X,\delta)$ then $F$ in a (locally)  $s$-set
%in $(X,d, \mu)$. Also, if $F$ satisfy a certain  property, then the reciprocal also holds.
%We prove that this property is satisfy by compact sets, and we exhibit examples of unbounded sets satisfying it.\\

The paper is organized as follows. Section~\ref{section: main results} contains the main results.
Theorem~\ref{Teo:s-conjunto implica mu conjunto} states that
under certain typical conditions, being $s$-set in $(X,\delta)$ is stronger
than being $s$-set in $(X,d,\mu)$. A sufficient condition under which every $s$-set in $(X,d,\mu)$ is an $s$-set in $(X,\delta)$
is contained in Theorem~\ref{teo: suficiente para s-set en (X,delta)}. We show that every bounded set satisfies this condition,
and we give examples of unbounded set satisfying it. In Proposition~\ref{propo: unicidad de la medida} we obtain a criteria to check the $s$-set condition related to $\mu$ of a given set based the Hausdorff measure.
Section~\ref{section: proof} is devoted to the proof of Proposition~\ref{propo: unicidad de la medida}, for which we state and proof a covering type lemma
of a bounded set by balls with small measure and controlled overlap (see Lemma~\ref{claim: cubriemiento por bolas de medida chica}).

\section{Main results}\label{section: main results}
Let $(X,d,\mu)$ be a given space of homogeneous type, and set $\delta$ the quasi-metric defined in previous section. We shall first prove that,
under certain condition, being $s$-set in $(X,\delta)$ is stronger
than being $s$-set in $(X,d,\mu)$.

\begin{theorem}\label{Teo:s-conjunto implica mu conjunto}\hspace*{\fill}
\begin{enumerate}
\item\label{item: s-set unbounded en delta} If $F$ is an unbounded $s$-set in $(X,\delta)$ with associated measure $\nu$,
then $F$  is  $s$-set in $(X,d,\mu)$ with the same measure $\nu$.
\item\label{item: loc s-set en delta}   If $F$ is locally $s$-set in $(X,\delta)$ with associated measure $\nu$ and $\mu(F)=0$,
then $F$  is locally $s$-set in $(X,d,\mu)$ with the same measure~$\nu$.
\end{enumerate}
\end{theorem}

\begin{proof}
By hypothesis, there exist  $c\ge 1$ and $r_0>0$  such that the inequalities
\[
c^{-1}r^s\leq \nu(B_\delta(x,r))\leq cr^s,
 \]
hold for every $x\in F$ and every $0<r< r_0$, where $\nu$ is a Borel measure supported in $F$, and
$r_0=\infty$ in case~(\ref{item: s-set unbounded en delta}).

Fix $x\in F$ and $r>0$.
By definition of $\delta$, we have that $B(x,r)\subseteq B_\delta (x, 2\mu(B(x,r)))$.
Then,
\[
 \nu\left(B(x,r)\right)\le \nu\left( B_\delta (x, 2\mu(B(x,r)))\right) \le c  2^s \mu^s\left(B(x,r)\right)
\]
provided that  $\mu(B(x,r)) <\frac{r_0}{2}$. On the other hand, fix $\ell$ such $3K^2\le 2^\ell$ where
$K$ denotes the triangular constant for $d$.
Then $B_\delta\left( x, A^{- \ell} \mu(B(x, r))\right) \subseteq B(x,r)$ (see \cite[pag. 262]{M-S}),
where $A$ is the constant for the doubling condition for $\mu$. Hence

\[\nu\left( B(x,r)\right)\geq \nu\left( B_\delta\left( A^{- \ell}\mu(B(x, r))  \right) \right)
 \geq c^{-1} A^{- \ell s}\mu^s (B(x, r)),
\]
provided that $ \mu(B(x, r)) <A^{ \ell}r_0$.

Since every $d$-ball has finite $\mu$-measure, (\ref{item: s-set unbounded en delta}) is proved. On the other hand, we obtain
(\ref{item: loc s-set en delta})
 if we can choice $r_1$ in such a way that $0<r<r_1$ implies $\mu(B(x, r)) <\min\{\frac{r_0}{2},A^{ \ell}r_0\}=\frac{r_0}{2}$,
for every $x\in F$.  But this is possible from the hypothesis
 $\mu(F)=0$.
\end{proof}

We shall point out that the assumption $\mu(F)=0$ is natural in many problems related with partial differential equations,
in which $F$ plays the role of the boundary of a domain in a metric measure space $(X,d,\mu)$
 (see
for example  \cite{DuranSanmartinoToschi} or \cite{DuranLopez}).\\

In order to obtain a sufficient condition under which every locally $s$-set in $(X,d,\mu)$
becomes a locally $s$-set in $(X,\delta)$, we shall give the following definition.

\begin{definition}
Let $F$ be a closed subset of $X$. We shall say that $F$ is \textbf{\emph{consistent with}} $\boldsymbol\mu$ if there exists a positive number $R$ such that
\begin{equation*}
\inf_{x\in F}\mu(B(x,R))>0.
\end{equation*}
\end{definition}

Let us remark that if $F$ is a set consistent with  $\mu$, then we have that $\inf_{x\in F}\mu(B(x,r))>0$ for every $r>0$. In fact, the claim is
trivial for every $r\geq R$. On the other hand,  for a fixed $0<r<R$, for every $x\in F$ we have that
\[\mu(B(x,r))=\mu\left(x, \frac{r}{R}R\right)\geq \frac{1}{A^m}\mu(B(x,R)),\]
where $m$ is a positive integer such that $2^m\geq R/r$ and $A$ denotes the doubling constant for $\mu$.\\

We want also to point out that every bounded subset of $X$ is consistent with $\mu$. In fact, set $R=2K\diam(F)$, with $K$ the triangular constant for $d$, and fix $x_0\in F$.
Then $B(x_0,\diam(F))\subseteq B(x,R)$ for every $x\in F$. Then $\inf_{x\in F}\mu(B(x,R))\geq \mu(B(x_0,\diam(F)))>0$, since $\mu$ is doubling.\\

However, there exist also
unbounded sets satisfying this condition.

\begin{example}
 Consider $X=\mathbb R^2$ equipped with the usual distance $d$
and the Lebesgue measure $\lambda$. Fix $a>0$ and set $F=\{(t,0): t\geq a\}$. Then $\lambda(B(x,r))$ is equivalent to $r^2$ for every $x\in F$, thus $F$ is consistent with $\lambda$.
\end{example}

\begin{example}\label{ex: no acotado}
Also we can consider another measure $\mu$ defined on $(\mathbb R^2,d)$ in such a way that $(X,d,\mu)$ is not an Ahlfors space. For example, let us consider
the measure $\mu$ define by
\[\mu(E)=\int_E |y|^{\beta} dy,\]
for a fixed $\beta>-2$. Then  $(X,d,\mu)$ is a space of homogeneous type since $|x|^{\beta}$ is a Muckenhoupt weight (see \cite{Muck} or  \cite{garcia-rubio}). For the set  $F$ considered in the above example,
 it is easy to see that $\mu(B(x,r))$ is equivalent to $r^2 |x|^\beta$ for $x\in F$ and $0<r\leq a/2$.
 So that $F$ is consistent with $\mu$ provided that $\beta>0$.
\end{example}

With this terminology, we have the following result.

\begin{theorem}\label{teo: suficiente para s-set en (X,delta)} \hspace*{\fill}
\begin{enumerate}
\item\label{item: s-set unbounded con mu} If $F$ is an unbounded $s$-set in $(X,d,\mu)$, then $F$ is  $s$-set in $(X,\delta)$.
\item\label{item: loc s-set con mu} If $F$ is a locally $s$-set in $(X,d,\mu)$ which is consistent with $\mu$, then $F$ is locally $s$-set in $(X,\delta)$.
\end{enumerate}
\end{theorem}

Let us observe that every bounded $s$-set in $(X,d,\mu)$ satisfies the hypothesis of the above theorem.
In order to prove this theorem, we shall need the following three auxiliary results.\\

The first one states that, as in the case of $s$-sets relative to a distance,
when $F$ is $s$-set relative to the measure $\mu$, there exists essentially only one Borel  measure $\nu$ satisfying the required condition. More precisely, we state the following result that
we shall prove  in Section~\ref{section: proof}.

\begin{proposition}\label{propo: unicidad de la medida}
If
$F$ is (locally) $s$-set in $(X,d,\mu)$ with measure $m$, then $F$ is (locally) $s$-set en $(X,d,\mu)$ with the restriction of  $H^s$ to~$F$,
where $H^s$ denotes the $s$-dimensional Hausdorff measure relative
to $\mu$.
\end{proposition}

The following statement is about a characterization of consistent sets, and says that the radii
 of all the $d$-balls  centering in a set consistent
with $\mu$  are as small as we want, provided that  the ball has sufficiently small measure.

\begin{lemma}\label{lemma: radios chicos}
$F$ is consistent with $\mu$ if and only if
given $r_0>0$, there exists $C$ such that if $x\in F$ and $\mu(B(x,t))\leq C$, then $t<r_0$.
\end{lemma}

\begin{proof}%[Proof of Lemma~\ref{lemma: radios chicos}]
Suppose first that $F$ is consistent with $\mu$ but the property  is false. Then there exists $r_0>0$ such that for every
natural number $n$ we can find $x_n\in F$ and $t_n\geq r_0$ with $\mu(B(x_n,t_n))\leq \frac 1n$.
So that $\mu(B(x_n,r_0))\leq \frac 1n$ for every natural $n$, which implies that
$\inf_{x\in F}\mu(B(x,r_0))=0$. But this is a contradiction, since $F$ is consistent with $\mu$.
Reciprocally, assume that $F$ is not consistent with $\mu$. Then, for every $r_0>0$ we have that $\inf_{x\in F}\mu(B(x,r_0))=0$.
So that for every natural $n$ there exists $x_n\in F$  such that $\mu(B(x_n,r_0))<\frac 1n$. Hence, given $C>0$ we can choose $n$ such that
$1/n\leq C$ and obtain $\mu(B(x_n,r_0))<C$ but $r_0\nless r_0$.
\end{proof}

The last  result that we shall need is a
technical lemma, which  is showed in \cite{Sjoding}, so that we shall omit its proof.
\begin{lemma}\label{Lema:Sjodin}
Given $x\in X$ and $0<r<2\mu(X)$,
there exist numbers $0<a\leq b<\infty$  such that
\begin{equation*}
 B(x,a)\subseteq B_\delta(x,r)\subseteq B(x,b)
\end{equation*}
and
\begin{equation*}
 C_1 r\leq \mu(B(x,a)) \leq \mu(B(x,b)) \leq C_2 r,
\end{equation*}
where  $C_1$ and $ C_2$  only depend on $X$.
\end{lemma}

\begin{proof}[Proof of Theorem~\ref{teo: suficiente para s-set en (X,delta)}]
From Proposition~\ref{propo: unicidad de la medida}, there exist  $c\ge 1$ and $r_0>0$ such that
  \begin{equation*}\label{eq: normal mu con hausdorff}
c^{-1}\mu(B(x,r))^s\leq H^s(B(x,r)\cap F)\leq c\mu(B(x,r))^s,
\end{equation*}
for every $x\in F$ and every $0<r<r_0$, where
$r_0=\infty$ in case~(\ref{item: s-set unbounded con mu}).

Fix $x\in F$ and $0<r<2\mu(X)$, and let $a$ and $b$ be
as in Lemma~\ref{Lema:Sjodin}. Then
\[
H^s\left(B_\delta(x,r)\cap F\right)
\leq  H^s\left(B(x,b)\cap F\right)
\leq c  \mu^s\left(B(x,b)\right)
\leq  c C_2^s r^s,
\]
and
\[
H^s\left(B_\delta(x,r)\cap F\right)
\geq  H^s\left(B(x,a)\cap F\right)
\geq c^{-1}  \mu^s\left(B(x,a)\right)
\geq  c^{-1} C_1^s r^s,
\]
provided that $a,b<r_0$. Then (\ref{item: s-set unbounded con mu}) is proved.
On the other hand, (\ref{item: loc s-set con mu}) is showed if we can choice $r_1\leq 2\mu(X)$ such that
$r<r_1$ implies $a,b<r_0$. In order to do this, let $C$ be such that if $x\in F$ and $\mu(B(x,t))\leq C$, then $t<r_0$
(see  Lemma~\ref{lemma: radios chicos}). Let us define $r_1=\min\{2\mu(X),C/C_2\}$, with $C_2$  the constant that
appears in Lemma~\ref{Lema:Sjodin}. Then  $\mu(B(x,a))$ and $\mu(B(x,b))$  are both bounded above by $C$, so that $a,b<r_0$.
\end{proof}

\begin{remark} We want  to point out that the condition  ``$F$ consistent with $\mu$'' in Theorem~\ref{teo: suficiente para s-set en (X,delta)}
is sufficient for a locally $s$-set in $(X,d,\mu)$ to be a locally $s$-set in $(X,\delta)$, but is not necessary.
In fact, let us consider $(X,d,\mu)$ and $F$ as in Example~\ref{ex: no acotado}. Taking
\[\nu(E)=\int_{E\cap F} |s|^{\beta/2} ds\]
as the Borel measure supported on $F$ we can show that $F$ is locally $\frac12$-set in $(X,\delta)$, and from
Theorem~\ref{Teo:s-conjunto implica mu conjunto} we have that $F$ is locally $\frac12$-set in $(X,d,\mu)$.
Nevertheless, it is easy to see that $F=\{(t,0): t\geq a\}$ is not consistent with $\mu$ if $\beta<0$.
\end{remark}

\section{Proof of Proposition~\ref{propo: unicidad de la medida}}\label{section: proof}

In order to prove Proposition~\ref{propo: unicidad de la medida}, we shall use the following covering type lemma that we shall
prove at the end of this section.

\begin{lemma}\label{claim: cubriemiento por bolas de medida chica}
Let $G$ be a bounded subset of $X$. For a given $\rho>0$,
there exists a finite covering $\{B(x_i,r_i), i=1,\dots,I_\rho\}$ of $G$
 by $d$-balls with $x_i\in G$  and $\mu(B(x_i,r_i))<\rho$.  Also, each $y\in X$ belongs to at most $\Lambda$ of such balls, where $\Lambda$ is a geometric
constant  which depends only on $X$.
 %does not depend on $\rho$ or $G$.
\end{lemma}

\begin{remark}\label{rem}
 Notice that if $\rho\leq\mu(G)$, then $r_i\leq \diam(G)$ for every $i$. In fact, let us assume that $r_i>\diam(G)$ for some $i$.
Then $G\subseteq B(x_i,r_i)$, so that $\mu(G)\leq \mu(B(x_i,r_i))<\rho\leq\mu(G)$, which is an absurd.
\end{remark}

\begin{proof}[Proof of Proposition~\ref{propo: unicidad de la medida}]
By hypothesis there exist  $r_0>0$, a constant $c\geq 1$ and
 a Borel measure $m$ supported on $F$ such that
  \begin{equation*}
c^{-1}\mu(B(x,r))^s\leq m(B(x,r))\leq c\mu(B(x,r))^s,
\end{equation*}
for every $x\in F$ and every $0<r<r_0$. Here $r_0$ is infinite if $F$ is an unbounded $s$-set in $(X,d,\mu)$, and is
finite otherwise.

Fix $x\in F$, $0<r<r_0$ and $\varepsilon>0$. For each $\rho>0$, there exists a covering $\{B_i=B(x_i,r_i)\}$ of $B(x,r)\cap  F$ by balls   such that
$\mu(B_i)<\rho$ and
\[
\sum_{i\ge 1}\mu^s(B_i)<{H_\rho}^s(B(x,r)\cap  F) +\varepsilon\leq {H}^s(B(x,r)\cap  F) +\varepsilon.
\]
Choosing an appropriated value of $\rho$, we can also obtain $r_i<r_0$ for every $i$. In fact, take
  $\rho=\mu(B(x,r))/A^\ell$ with $\ell$ an integer such that $2^\ell\geq 3K^2$. Then, since we can assume that each $B(x_i,r_i)$ intersects $B(x,r)$,
  if some $r_i\geq r_0$ then we have that $B(x,r)\subseteq B(x_i,3K^2r_i)$. Hence $\mu(B(x,r))\leq A^\ell \mu(B_i)<\mu(B(x,r))$, which is absurd.
Then we can assume $r_i<r_0$ for every $i$,  and hence
\[
c^{-1} \mu(B(x,r))^s  \leq m(B(x,r))\leq \sum_{i} m(B_i) \leq c\sum_{i} \mu(B_i)^s.
\]
Hence, $c^{-1} \mu(B(x,r))^s <cH^s(B(x,r)\cap  F) +c\varepsilon$ for every $\varepsilon>0$, which proves that
\[H^s(B(x,r)\cap  F)\geq c^{-2} \mu(B(x,r))^s. \]

In order to obtain an upper bound for  $H^s(B(x,r)\cap  F)$, let us first assume that $r<\frac{r_0}{4K^2}$ and we fix $0<\rho<\mu(B(x,r)\cap  F)$.
From Lemma~\ref{claim: cubriemiento por bolas de medida chica},
there exists a finite covering $\{B(x_i,r_i), i=1,\dots,I_\rho\}$ of $B(x,r)\cap  F$
 by $d$-balls satisfying  $\mu(B(x_i,r_i))<\rho$, $x_i\in F$ and $r_i\leq 2Kr$. Also, each $y\in X$ belongs to at most $\Lambda$ of such balls, where $\Lambda$ is a geometric
constant  which does not depend on
$\rho$, $r$ or $x$.
 So, we have that
\begin{align*}
    H_{\rho}^s(B(x,r)\cap  F)&\leq \sum_{i=1}^{I_\rho} \mu(B(x_i,r_i))^s
\\&
\leq c \sum_{i=1}^{I_\rho} m\left(B(x_i,r_i)\right)\\
   &\leq c \Lambda m\left(\bigcup_{i=1}^{I_\rho} B(x_i,r_i)\right)
\\ &
\leq c \Lambda  m  \left(B(x,4K^2r)\right)
\\&
\leq c^2 \Lambda  \mu(B(x,4K^2r))^s
%\\&
%\le c^2 \Lambda A^j \mu(B(x,r))^s,
\\&
= \tilde C \mu(B(x,r))^s,
 \end{align*}
with $\tilde C=c^2 \Lambda A^j$, where $j$ is a positive integer such that $2^{j-2}\geq K^2$. Taking $\rho\to 0$ we obtain the desired result for this case.

Finally, if $r_0$ is finite, we shall consider the case $\frac{r_0}{4K^2}\leq r<r_0$. In this case, since
 $B(x,r)$ is bounded,  there exists a finite $r_0 (8K^2)^{-1}$-disperse maximal set in $B(x,r)$, let us say $U=\{x_1,\dots, x_{I}\}$, with $I\leq N^{2+\log_2K}$.
Then $B(x,r)\cap F\subseteq \bigcup_{i=1}^{I} B\left(x_i,\frac{r_0}{8K^2}\right)$, and applying the previous case we obtain
 \[H^s(B(x,r)\cap F)\leq \sum_{i=1}^{I}  H^s\left(
 B\left(x_i,\frac{r_0}{8K^2}\right)\cap F\right)\leq \tilde C I  \mu\left(B\left(x,2Kr\right)\right)^s,  \]
and the result follows from the doubling property of $\mu$.
\end{proof}

For the proof of
Lemma~\ref{claim: cubriemiento por bolas de medida chica}, we shall use the next result
about the behavior of  $\delta$-diameter $\diam_\delta(E):=\sup\{\delta(y,w):y,w\in E\}$ of a bounded set $E$.

\begin{lemma}\label{lemma: mu es como el diametro}
Let $E$ be a bounded subset of $X$. For $B=B(x,\diam(E))$ and $x\in E$ we have
\[A^{-\ell}\mu(B)\leq\diam_\delta(E)\leq A\mu(B),\]
where $A$ is the doubling constant for $\mu$ and $\ell$ is a positive integer satisfying $\ell\geq \log_2(8K^3)$, with  $K$ the triangular constant for $d$.
%Here $\diam_\delta(E):=\sup\{\delta(y,w):y,w\in E\}$.
\end{lemma}

\begin{proof}%[Proof of Lemma~\ref{lemma: mu es como el diametro}]
Let us fix $x\in E$, and let $y$ and $w$ any two points in $E$. Since $y,w\in B(x,2\diam(E))$, from the definition of $\delta$ follows that $\delta(y,w)\leq \mu(B_d(x,2\diam(E)))\leq A\mu(B)$. Taking supreme the upper bound for $\diam_\delta(E)$ is obtained.

For the lower bound, let $y_0,w_0\in E$ such that $\diam(E)<2d(y_0,w_0)$. For a given $\varepsilon>0$, let $B(x_0,r_0)$ be a ball containing  $y_0$ and $w_0$ such that $\mu(B(x_0,r_0))<\delta(y_0,w_0)+\varepsilon$. We claim that $B\subseteq B(x_0,8K^3r_0)$. Assuming this fact true we have that
\[\diam_\delta(F)\geq \delta(y_0,w_0)> \mu(B(x_0,r_0))-\varepsilon
\geq  A^{-\ell}\mu(B)-\varepsilon.\]
By letting $\varepsilon$ tends to zero we obtain the result. Only remains to prove the claim, for which fix $z\in B$. Then
\begin{align*}
  d(z,x_0)&\leq K^2[d(x,x)+d(x,w_0)+d(w_0,x_0)]\\
  &< K^2[2\diam(E)+r_0] \\
  &< K^2[4d(y_0,w_0)+r_0]\\
  &<K^2[4K(d(y_0,x_0)+d(x_0,w_0))+r_0]\\
  &<8K^3r_0,
\end{align*}
and the lemma is proved.
\end{proof}

\begin{proof}[Proof of Lemma~\ref{claim: cubriemiento por bolas de medida chica}]
 Let us denote $\tilde K$ the triangular constant for $\delta$ and $\tilde N$ the constant for the finite
metric dimension of $(X,\delta,\mu)$.
Given $\rho>0$, let
$t=\frac{\rho}{4\tilde K A^{\ell+1}}$, with $\ell$ as in Lemma~\ref{lemma: mu es como el diametro}. Set $U=\{x_1,\dots,x_{I_t}\}$ a finite
$t$-disperse maximal set in $G$ with respect to the quasi-metric $\delta$. So that $\{B_\delta(x_i,t)\}$ is a covering of $G$.
Let us define $B_i=B(x_i,r_i)$, with  $r_i=2\diam (B_\delta(x_i,t))$. Let us first check that $\{B_i\}$ is covering of $G$. In fact, if $y\in G$ then there exists
$i$ such that $y\in B_\delta(x_i,t)$. Then \[d(x_i,y)\leq \diam\left(B_\delta(x_i,t)\right)<2\diam\left(B_\delta(x_i,t)\right),\] so that $y\in B_i$.
In order to estimate the measure of each $B_i$,
 using Lemma~\ref{lemma: mu es como el diametro} with $E=B_\delta(x_i,t)$ we obtain
 \[\mu(B_i)\leq A\mu \left(B(x_i,\diam (B_\delta(x_i,t)))\right)\leq A^{\ell+1}\diam_\delta(B_\delta(x_i,t))\leq A^{\ell+1} 2\tilde K t.\]
From the choice of $t$, we have $\mu(B_i)<\rho$. So that it only remains to prove that we can control
the overlapping of this balls by a geometric constant~$\Lambda$. In fact, for a fixed $y\in X$ we have that  if $y\in B(x_i,r_i)$, then
$B(y,r_i)\subseteq B(x_i,2Kr_i)$. So that $\mu(B(y,r_i))\leq A^p \rho$, with $p$ and integer such that $2^{p-1}\geq K$, and then
\[x_i\in B(y,r_i)\subseteq B_\delta(y,2\mu(B(y,r_i)))\subseteq B_\delta(y,2A^p\rho)=B_\delta(y,8\tilde KA^{\ell+p+1}t).\]
Hence, the number of balls $B(x_i,r_i)$ to which $y$ belongs
is less than or equal to the cardinal of $U \cap B_\delta(y, 2^mt)$, with $m$ a natural number such that $2^m\geq 8\tilde KA^{\ell+p+1}$.
Since  $U$ is $t$-disperse with respect to $\delta$, we have that $\Lambda\leq \tilde N^m$
and the lemma is proved.
\end{proof}

\def\cprime{$'$}


\begin{thebibliography}{10}

%%%%%%% To ease editing, use normal size:

\normalsize
\baselineskip=17pt

%%%%%%%%%%%%%%%

\bibitem{ACDT}
Hugo Aimar, Marilina Carena, Ricardo Dur{\'a}n, and Marisa Toschi.
\newblock Powers of distances to lower dimensional sets as {M}uckenhoupt
  weights.
\newblock preprint.

\bibitem{Asso}
Patrice Assouad.
\newblock \'{E}tude d'une dimension m\'etrique li\'ee \`a la possibilit\'e de
  plongements dans {${\bf R}\sp{n}$}.
\newblock {\em C. R. Acad. Sci. Paris S\'er. A-B}, 288(15):A731--A734, 1979.

\bibitem{CoifWeiss}
Ronald~R. Coifman and Guido Weiss.
\newblock {\em Analyse harmonique non-commutative sur certains espaces
  homog\`enes}.
\newblock Springer-Verlag, Berlin, 1971.
\newblock \'Etude de certaines int\'egrales singuli\`eres, Lecture Notes in
  Mathematics, Vol. 242.

\bibitem{DuranSanmartinoToschi}
R.~G. Dur{\'a}n, M.~Sanmartino, and M.~Toschi.
\newblock Weighted a priori estimates for the {P}oisson equation.
\newblock {\em Indiana Univ. Math. J.}, 57(7):3463--3478, 2008.

\bibitem{DuranLopez}
Ricardo~G. Dur{\'a}n and Fernando L{\'o}pez~Garc{\'{\i}}a.
\newblock Solutions of the divergence and analysis of the {S}tokes equations in
  planar {H}\"older-{$\alpha$} domains.
\newblock {\em Math. Models Methods Appl. Sci.}, 20(1):95--120, 2010.

\bibitem{Falconer}
K.~J. Falconer.
\newblock {\em The geometry of fractal sets}, volume~85 of {\em Cambridge
  Tracts in Mathematics}.
\newblock Cambridge University Press, Cambridge, 1986.

\bibitem{garcia-rubio}
Jos{\'e} Garc{\'{\i}}a-Cuerva and Jos{\'e}~L. Rubio~de Francia.
\newblock {\em Weighted norm inequalities and related topics}, volume 116 of
  {\em North-Holland Mathematics Studies}.
\newblock North-Holland Publishing Co., Amsterdam, 1985.
\newblock , Notas de Matem\'atica [Mathematical Notes], 104.

\bibitem{M-S}
Roberto~A. Mac{\'{\i}}as and Carlos Segovia.
\newblock Lipschitz functions on spaces of homogeneous type.
\newblock {\em Adv. in Math.}, 33(3):257--270, 1979.

\bibitem{Muck}
Benjamin Muckenhoupt.
\newblock Weighted norm inequalities for the {H}ardy maximal function.
\newblock {\em Trans. Amer. Math. Soc.}, 165:207--226, 1972.

\bibitem{Sjoding}
Tord Sj{\"o}din.
\newblock On {$s$}-sets and mutual absolute continuity of measures on
  homogeneous spaces.
\newblock {\em Manuscripta Math.}, 94(2):169--186, 1997.

\bibitem{Triebel}
Hans Triebel.
\newblock {\em Fractals and spectra}.
\newblock Modern Birkh\"auser Classics. Birkh\"auser Verlag, Basel, 2011.
\newblock Related to Fourier analysis and function spaces.

\end{thebibliography}
\end{document}